\documentclass[reqno]{amsart}
\usepackage{amsfonts}%
\usepackage{amssymb}
\usepackage{mathrsfs}
\usepackage{latexsym,amsmath,amsfonts}
\usepackage[pagebackref]{hyperref}
\usepackage{cite}
\usepackage{color}
\allowdisplaybreaks
\date{\today}
\theoremstyle{plain}
\newtheorem{thm}{Theorem}[section]

\newtheorem{lem}[thm]{Lemma}
\newtheorem{prop}[thm]{Proposition}
\theoremstyle{definition}

\theoremstyle{remark}
\newtheorem{rem}{Remark}[section]
\numberwithin{equation}{section}

\def\be{\begin{eqnarray}}
\def\ee{\end{eqnarray}}
\def\ft{\frac d{dt}}
\def\pt{\partial_t}
\def\r{{\mathbb{T}}^3}

\def\md{\mbox{div}}

\def\c{\mbox{curl\,}}
\begin{document}
% \linenumbers

\title[Quasi-neutral limit of Euler-Poisson system coupled to a magnetic field] {Quasi-neutral
limit of Euler-Poisson system of compressible fluids coupled to a
magnetic field}

\author{Jianwei Yang}
\address{
 College of  Mathematics and Information Science,
North China University of  Water Resources and Electric Power,
                                      Zhengzhou 450045, Henan Province, P. R.
                                      China}
 \email{yangjianwei@ncwu.edu.cn}



\begin{abstract}

  In this paper, we consider the quasi-neutral limit of a three dimensional
Euler-Poisson  system of compressible fluids coupled to a magnetic
field. We prove that, as Debye length tends to zero, periodic
initial-value problems of the model have unique smooth solutions
existing in the time interval  where the ideal incompressible
magnetohydrodynamic equations has smooth solution. Meanwhile, it is
proved that smooth solutions converge to solutions of incompressible
magnetohydrodynamic equations with a sharp convergence rate in the
process of quasi-neutral limit.

\end{abstract}

\maketitle

\section{ Introduction }

The main objective of this paper is to study the quasi-neutral limit
of the following Euler-Poisson system of compressible fluids coupled
to a magnetic field \cite{fls15,st04}:
\begin{align}
             &\pt n^{\lambda}+\md (n^{\lambda}u^{\lambda})=0 \:\:\mbox{in}\:\:\r, t>0,\label{epm1}\\
             &\displaystyle\pt (n^{\lambda}u^{\lambda})+\md(n^{\lambda}u^{\lambda}\otimes u^{\lambda})
             +\nabla p(n^{\lambda})
 =n^{\lambda}\nabla\phi^{\lambda}+\c B^{\lambda}\times B^{\lambda}, \label{epm2}\\
&\pt B^{\lambda}-\c(u^{\lambda}\times B^{\lambda})=0,\label{epm3}\\
&\md B^{\lambda}=0,\label{epm4}\\
&\lambda^2\Delta\phi^{\lambda}=n^{\lambda}-1,\label{epm5}
\end{align}
with initial conditions:
\begin{equation}\label{epmi}
    n^{\lambda}(\cdot,0)=n^{\lambda}_{0}, u^{\lambda}(\cdot,0)=u^{\lambda}_{0}, B^{\lambda}(\cdot,0)
    =B^{\lambda}_{0}.
\end{equation}
In the above equations, $\r$ is 3-dimensional torus and $\lambda>0$
is
 the (scaled)  Debye length.
  The unknown functions are the density  $n^{\lambda}$,
the velocity $u^{\lambda} =
(u^{\lambda}_1,u^{\lambda}_2,u^{\lambda}_3)$, the magnetic field
$B^{\lambda} = (B^{\lambda}_1,B^{\lambda}_2,B^{\lambda}_3)$ and the
gravitational potential $\phi^{\lambda}$. Throughout this paper, we
assume that the pressure function $p(n^{\lambda})$ satisfies the
usual $\gamma-$law,
$$p(n^{\lambda})=\frac{(n^{\lambda})^\gamma}{\gamma},
n^\lambda>0,$$ for  some constant $\gamma>1$. It is obvious that
equations\eqref{epm4} is redundant with equations \eqref{epm3}, as
soon as they are satisfied by the initial conditions $\md
B^{\lambda}_0=0$. System \eqref{epm1}-\eqref{epm5} is used to model
the evolution of a magnetic stars \cite{cg68}. The effects of
magnetic fields arise in some physically interesting and important
phenomena in astrophysics; e.g. solar flares.  Without taking
magnetic effects into account, system \eqref{epm1}-\eqref{epm5}
reduces to the Euler-Poisson equations.

In recent years,  the  quasi-neutral  limit ($\lambda\rightarrow0$)
of various models has attracted much attention. In particular, the
limit $\lambda\rightarrow0$ has been performed in Vlasov-Poisson
system by Brenier \cite{b00}, Grenier \cite{g95,g96,g97} and
Masmoudi \cite{m01}, in drift-diffusion equations by Gasser et al.
\cite{glms01,ghmw02} and J\"{u}ngel and Peng \cite{jp01}, and in the
one dimensional and isothermal Euler-Poisson system by Cordier and
Grenier \cite{cg00}, in more general isentropic models by
 Wang \cite{w04}, in
non-isentropic Euler-Poisson equations by Peng et al. \cite{pwy06}
and Li \cite{l08}, in Euler-Monge-Amp\`{e}re systems by Loeper
\cite{l05},   in Navier-Stokes-Poisson system by Wang and Jiang
\cite{ws06}, Donatelli and Marcati \cite{dm12} and Ju et al.
\cite{jll09}, in quantum hydrodynamics equations \cite{ll05}, in
Navier-Stokes-Fourier-Poisson system by Li et al.\cite{ljx15}, etc.
As far as we know, there is no result on quasi-neutral limit of the
Euler-Poisson system coupled to a magnetic field
\eqref{epm1}-\eqref{epm5}.

In this paper, we will study the quasi-neutral limit for the smooth
solution of the system \eqref{epm1}-\eqref{epm5} in the framework of
the convergence-stability principle developed in \cite{y2001}.
Formally, taking the (scaled)  Debye length $\lambda\rightarrow0$ in
\eqref{epm5}, we obtain the following ideal incompressible
magnetohydrodynamic equations
\begin{align}
  &\pt u^0+(u^0\cdot\nabla)u^0
             +\nabla p^0=\c B^0\times B^0,\label{mhd1}\\
 & \pt B^0-\c(u^0\times B^0)=0,\label{mhd2}\\
 & \md u^0=\md
B^0=0.\label{mhd3}
\end{align}

The objective of this paper is to make this limit rigorous. Our
proof requires the (local) existence of a smooth solution to
\eqref{epm1}-\eqref{epm5}, which is shown in next section. The proof
of our result is based on the convergence-stability principle
developed by Yong \cite{y99,y2001} for singular limit problems of
symmetrizable hyperbolic systems.
 In
contrast with the results in
\cite{b00,cg00,dm12,jll09,ll05,l08,ljx15,l05,w04,ws06}, where the
limiting equations are incompressible Euler equations or the
incompressible Navier-Stokes equations, our limiting equations are
the incompressible magnetohydrodynamic equations
\eqref{mhd1}-\eqref{mhd3}. In our case where the Euler-Poisson
equations are coupled to a magnetic field, the problem becomes more
challenging.  Because of the magnetic field and non-linearity terms,
some elaborated energy analysis are required to obtain the desired
convergence results.

This paper is organized as follows. In section 2, we rewrite the
system \eqref{epm1}-\eqref{epm5} as a symmetrizable hyperbolic
system to obtain the local-in-time existence result, and present our
main results.  The proof of Theorem \ref{thm} is obtained in section
3.

Before ending the introduction, we give the notation and Lemma used
throughout the current paper. The letters $C$ and $C_T$ denote
various positive constants independent of $\lambda$, which can be
different from one line to another one, but $C_T$  may depend on
$T$. The symbol \lq\lq:\rq\rq means summation over both matrix
indices. $|U|$   denotes some norm of a vector or matrix $U$. Also,
we denote
$$\|\cdot\|=\|\cdot\|_{L^2(\r)},\,\,\|\cdot\|_{\infty}=\|\cdot\|_{L^\infty(\r)},\,\,
\|\cdot\|_{k}=\|\cdot\|_{H^k(\r)}, \,\,k\in\mathbb{N}^{\ast}.$$

\begin{lem}(\emph{see, e.g.\cite{majda}}). Let $s, s_1$, and $s_2$ be three nonnegative integers and $s_0
=\left[\frac{d}{2}\right] + 1$.

\emph{1.} If $f,  g\in H^s\cap L^\infty$ and any nonnegative multi
index $\beta, |\beta|\le s$, then we have \be \|D^\beta (fg)\|\le
C_s (\|f\|_{L^\infty}\|D^\beta g\|+\|g\|_{L^\infty}\|D^\beta
f\|)\leq c_s\|f\|_s\|g\|_s.\notag\ee

\emph{2.} If $f\in H^s$, $Df\in L^\infty$,$g\in H^{s-1}\cap
L^\infty$, then we have \be\|D^\beta (fg)-fD^\beta g\|\le C_s (\|D
f\|_{L^\infty}\|D_x^{\beta'}g\|+\|g\|_{L^\infty}\|D^\beta f\|),\:
|\beta'|=|\beta|- 1. \notag\ee

\emph{3.} Let $s_3=\min\{s_1,s_2,s_1+s_2-s_0\}\geq0$, then
$H^{s_1}H^{s_2}\subset H^{s_3}$. Here the inclusion symbol $\subset$
implies the continuity of the embedding.

\emph{4.}   Suppose  $s\geq s_0$, $A\in C_b^s(G)$, and $U\in
H^s(\Omega,G)$. Then   $A(U(\cdot))\in H^s$ and
$$\|A(U(\cdot))\|_s\leq C_s|A|_s(1+\|U\|^s_s).$$
Here and below, $C_s$ denotes a generic constant depending only on
$s$ and $d$, and $|A|_s$ stands for $\sup\limits_{\{U\in
G,|\alpha|\leq s\}}|\partial^{\alpha}_{U}A(U)|$.
\end{lem}

\section{Main Results}
%\subsection{Local existence of smooth solution}
First, we consider the local existence of smooth solution of the
system \eqref{epm1}-\eqref{epm5} for any fixed $\lambda>0$.

By Green's formulation, it follows from \eqref{epm1} and
\eqref{epm5} that
\begin{align}\label{phi}
   \nabla\phi^\lambda=\frac{1}{\lambda^2}\left(\nabla\Delta^{-1}(n^\lambda_0-1)
   -\nabla\Delta^{-1}\md\int_0^t(n^\lambda u^\lambda)(x,\tau)d\tau\right).
\end{align}
Using \eqref{epm1} and the following equality
$$\md(n^\lambda u^\lambda\otimes u^\lambda)=n^\lambda(u^\lambda\cdot\nabla)u^\lambda
+u^\lambda\md(n^\lambda u^\lambda),$$ we can rewrite \eqref{epm2} as
\begin{align}\label{epm22}
   \pt u^\lambda+(u^\lambda\cdot\nabla)u^\lambda+\nabla h(n^\lambda)
 =\nabla\phi^\lambda+\frac{1}{n^\lambda}\c B^\lambda\times B^\lambda,
\end{align}where the enthalpy $h(n^\lambda) > 0$  is defined by
$$h'(n^\lambda)=\frac{p'(n^\lambda)}{n^\lambda}$$ for $n^\lambda>0$.

Set
$$ W^\lambda=\left(
                      \begin{array}{c}
                        n^\lambda \\
                        u^\lambda\\
                        B^\lambda \\
                      \end{array}
                    \right), W^\lambda_0=\left(
                      \begin{array}{c}
                        n^\lambda_0 \\
                        u^\lambda_0\\
                        B^\lambda_0 \\
                      \end{array}
                    \right),
$$
$$A_i(W^\lambda)=\left(
           \begin{array}{ccccccc}
             u^\lambda_i & n^\lambda e_i^T & 0 \\
             h'(n^\lambda)e_i & u^\lambda_i\textbf{I}_{3\times3} & \frac{(G^\lambda_i)^T}{n^\lambda}  \\
             0 & \frac{G^\lambda_i}{n^\lambda} & u^\lambda_i\textbf{I}_{3\times3}\\
           \end{array}
         \right), $$
         $$F^\lambda=\left(
                      \begin{array}{c}
                        0 \\
                        \nabla\Delta^{-1}(n^\lambda_0-1)-\nabla\Delta^{-1}
                        \md\int_0^t(n^\lambda  u^\lambda)(x,\tau)d\tau \\
                        0 \\
                      \end{array}
                    \right),
         $$
where $(e_1,e_2,e_3)$ is the canonical base of $\mathbb{R}^3$,
$\textbf{I}_{3\times3}$  is a  unit matrix, $y_i$ denotes the $i$th
component of $y\in\mathbb{R}^3$  and
$$G^\lambda_1=\left(
\begin{array}{ccc}
0 & 0 & 0  \\
 B^\lambda_2 & -B^\lambda_1 & 0  \\
 B^\lambda_3 & 0 & -B^\lambda_1 \\
           \end{array}
         \right),G^\lambda_2=\left(
\begin{array}{ccc}
-B^\lambda_2 & B^\lambda_1 & 0  \\
0 & 0 & 0  \\
0 &  B^\lambda_3 & -B^\lambda_2 \\
           \end{array}
         \right),$$$$G^\lambda_3=\left(
\begin{array}{ccc}
-B^\lambda_3 & 0 & B^\lambda_1  \\
 0 & -B^\lambda_3 &  B^\lambda_2  \\
 0 & 0 & 0 \\
           \end{array}
         \right).$$
Thus the problem \eqref{epm1}-\eqref{epmi} for the unknown $W$ can
be rewritten as \begin{align}
   &\pt W^\lambda+\sum\limits_{i=1}^{3}A_i(W^\lambda)\partial_{x_i}W^\lambda
   =\frac{1}{\lambda^2}F^\lambda,\label{ww1}\\
    &W^\lambda|_{t=0}=W^\lambda_0.\label{ww2}
\end{align}

It is not difficult to see that the equations of $W^\lambda$ in
\eqref{ww1}-\eqref{ww2} are symmetrizable hyperbolic, i.e. if we
introduce
$$A_0(n^\lambda)=\left(
           \begin{array}{ccccccc}
             h'(n^\lambda) & 0 & 0 \\
             0 & n^\lambda\textbf{I}_{3\times3} & 0  \\
             0 & 0 &n^\lambda \textbf{I}_{3\times3} \\
           \end{array}
         \right),$$
which is positively definite when $n^\lambda>\delta>0$, then
$\widehat{A}_i(W^\lambda)=A_0(n^\lambda)A_i(W^\lambda)$ are
symmetric for all $1\leq i\leq3$.

Thus, to solve the system \eqref{epm1}-\eqref{epmi}, it suffices to
solve the system \eqref{ww1}-\eqref{ww2}. Since the non-local source
term $\nabla\Delta^{-1}\md\int_0^t(n^\lambda
u^\lambda)(x,\tau)d\tau$ is a sum of products of Riesz transforms of
$\int_0^t(n^\lambda u^\lambda)(x,\tau)d\tau$, we have, by the $L^2$
boundedness of the Riesz transformation (see \cite{s70}),
$$\left\|\nabla\Delta^{-1}\md\int_{0}^{t}(n^\lambda u^\lambda)(x,\tau)d\tau\right\|_{s}
\leq C\left\|\int_{0}^{t}(n^\lambda
u^\lambda)(x,\mu)d\mu\right\|_{s},$$ for some constant $C > 0$
independent of $t$.

Moreover, we recall the other elementary fact which can be easily
proven by using Fourier series.
\begin{lem}
$\nabla\Delta^{-1}$  is a bounded linear operator from $V=\{v\in
L^2(\r)| \textbf{m}(v)=0\}$ into $H^1(\r)$.
\end{lem}
So one gets
$$\|\nabla\Delta^{-1}(n^\lambda_{0}-1)\|_{s+1}\leq C\|(n^\lambda_{0}-1)\|_s$$
for some constant $C > 0$.  Based on the above  crucial facts, using
the standard iteration techniques of local existence theory for
symmetrizable hyperbolic systems (see \cite{majda}), we have
\begin{prop}\label{prop18}Assume that the initial conditions $(n^\lambda_0,u^\lambda_0,B^\lambda_0)\in
H^s(\r)$, $s>\frac52, n_0>\delta>0, \emph{\md} B^\lambda_0=0$ and
$\int_\r(n^\lambda_0-1)dx=0$. Then for any fixed $\lambda>0$, there
exists a positive constant $T$ (may depend on $\lambda$) such that
the periodic problem \eqref{ww1}-\eqref{ww2} has a unique smooth
solution $(n^\lambda, u^\lambda, B^\lambda)\in C([0,T];H^s(\r))$,
well-defined on $\r\times[0,T]$.
 Hence the
nonlinear periodic problem \eqref{epm1}-\eqref{epmi} admits a unique
solution $(n^\lambda, u^\lambda,\nabla\phi^\lambda,B^\lambda)$
satisfying
\begin{equation}(n^\lambda, u^\lambda,\nabla\phi^\lambda, B^\lambda)\in C([0,T];H^s(\r)).
\label{190}\end{equation}
\end{prop}
According to Proposition \ref{prop18}, for each fixed $\lambda>0$ in
\eqref{epm1}-\eqref{epmi}, there exists a time interval $[0,T]$ such
that system \eqref{epm1}-\eqref{epmi} has a unique solution
$(n^\lambda, u^\lambda,\nabla\phi^\lambda,B^\lambda)$ satisfying
\eqref{190}. Define
\begin{equation}\label{define1}
  \begin{split} T_\lambda&=\sup \,\Big\{T>0: (n^\lambda, u^\lambda,
   \nabla\phi^\lambda, B^\lambda)\in C([0,T];H^s(\r)),\\
   &  \frac12\leq n^\lambda \leq\frac32, \forall (x,t)\in
   \r\times[0,T]\Big\}.\end{split}
\end{equation}
Namely, $[0, T_\lambda]$ is the maximal time interval of $H^s-$
existence. Note that $T_\lambda$  may tend to zero as $\lambda$ goes
to  $0$. In order to show that
$\liminf\limits_{\lambda\rightarrow0}T_\lambda > 0$, we follow the
convergence-stability principe \cite{y2001} and   seek a formal
approximation of $(n^\lambda, u^\lambda, B^\lambda)$. To this end,
we consider the initial-value problem  of the ideal incompressible
magnetohydrodynamic equation \eqref{mhd1}-\eqref{mhd3} with initial
data
\begin{align}\label{mhdi}
   u^{0}(\cdot,0)=u^{0}_{0}, B^{0}(\cdot,0)
    =B^{0}_{0}.
\end{align}

Let us  recall the local existence of a strong solution to the ideal
incompressible magnetohydrodynamic equations
\eqref{mhd1}-\eqref{mhd3}. The proof can be found in \cite{dl02}.

\begin{prop} \label{local}
(\emph{see \cite{dl02}.}) Let $s >\frac52$ be an integer. Assume
that the initial data $$(u^0(x, t),B^0(x, t))|_{t=0} = (u^0_0(x),
B^0_0(x))$$ satisfy
$$(u^0_0,B^0_0)\in H^{s+1}(\r) \:\:\:\emph{\mbox{and}}\:\:\:\emph{\md} u^0_0 = 0, \emph{\md} B^0_0 = 0.$$
 Then, there exist a $T_\ast\in(0,+\infty)$ and a unique solution
 $(u^0,B^0)\in L^\infty([0, T_\ast),H^{s+1}(\r))$ to the
ideal incompressible magnetohydrodynamic equations
\eqref{mhd1}-\eqref{mhd3} satisfying, for any $0 < T < T_\ast$,
$$\emph{\md} u^0 = 0, \emph{\md} B^0 = 0$$ and
\begin{eqnarray}
 \sup_{0\leq t\leq
 T}\Big(\|(u^0,B^0)\|_{s+1}+\|(\pt{u^0},\pt B^0)\|_{s}+\|
\nabla p^0\|_{s+1} +\|\pt\nabla p^0\|_{s}\Big)\leq C_T \label{ie}
\end{eqnarray} for some positive constant $C_T$.
\end{prop}

Now the main result of this paper reads as follows.

\begin{thm}\label{thm} Let $s > \frac32+ 1$ be an integer.  Suppose $\emph{\md} u^0_0=0,$ $\emph{\md} B^0_0=0$,
and incompressible magnetohydrodynamic equations
\eqref{mhd1}-\eqref{mhd3} with the initial data $(u^0_0,B^0_0)$ has
a solution $(u^0,B^0)\in L^\infty([0, T_\ast),H^{s+1}(\r))$. Then,
for $\lambda$ sufficiently small, there exists a $\lambda-$independent  positive number $T_{\ast\ast}<T\ast$, such that the  model
\eqref{epm1}-\eqref{epm5} with periodic initial data
$(n^\lambda_0,u^\lambda_0,B^\lambda_0)$ satisfying
\begin{align}\label{thmini}
   n^\lambda_0=1,  u^\lambda_0=u^0_0, B^\lambda_0=B^0_0,
\end{align}
has a unique solution
$(n^\lambda,u^\lambda,B^\lambda,\nabla\phi^\lambda)\in
C([0,T_{\ast\ast}];H^s(\r))$. Moreover, there exists a $\lambda-$independent  constant $M>0$ such that
\begin{align}\label{thm191}
    \sup_{0\leq t\leq
 T}(\|n^\lambda-1\|_s+\|u^\lambda-u^0\|_s+\|B^\lambda-B^0\|_s)\leq
   M\lambda.
\end{align}
\end{thm}
\begin{rem} The initial data
\begin{align*}
   n^\lambda_0=1,  u^\lambda_0=u^0_0, B^\lambda_0=B^0_0
\end{align*} can be relaxed as
\begin{align*}
   n^\lambda_0=1+O(\lambda^2),  u^\lambda_0=u^0_0+O(\lambda),
   B^\lambda_0=B^0_0+O(\lambda),
\end{align*} without changing our arguments.
\end{rem}

\begin{rem}
   Theorem \ref{thm} describes the quasi-neutral limit $\lambda\rightarrow0$ of the
    system \eqref{epm1}-\eqref{epm5} with well-prepared initial data, avoiding the
presence of the initial time layer. We will discuss the case of
general initial data (ill-prepared initial data) allowing the
presence of the fast singular oscillation in the future.
\end{rem}

\section{Proof of Theorem \ref{thm}}

 Thanks to the convergence-stability principle developed in \cite{y99,y2001},
 it suffices to prove the
error estimate in \eqref{thm191} for
$t\in[0,\min\{T_\lambda,T_{\ast\ast}\}]$ with $T_{\ast\ast}<T_{\ast}$ independent of $\lambda$ and to be determined. Thus we directly make the error
estimate \eqref{thm191} in the time interval
$[0,\min\{T_\lambda,T_{\ast\ast}\}]$.

Now we rewrite \eqref{ww1} as the following form
\begin{align}
    &\pt W^\lambda+\sum\limits_{i=1}^{3}A_i(W^\lambda)\partial_{x_i}W^\lambda
    =\left(
                   \begin{array}{c}
                     0 \\
                     \nabla\phi^\lambda  \\
                     0 \\
                   \end{array}
                 \right),\label{mep1}\\
    &\lambda^2\Delta\phi^{\lambda}=n^\lambda-1,\md B^\lambda=0.\label{mep2}
\end{align}
%where $A_i(W^\lambda)=(D_0(n^\lambda))^{-1}D_i(W^\lambda)$.

We note that with $(u^0,p^0,B^0)$ constructed in Proposition
\ref{local},
$$(n_\lambda,u_\lambda,\phi_\lambda,B_\lambda)=(1,u^0,-p^0,B^0)$$
 satisfies
\begin{align}
             &\pt n_{\lambda}+\md (n_{\lambda}u_{\lambda})=0,\label{mhd2111}\\
             &\displaystyle  \pt u^{\lambda}+(u_{\lambda}\cdot
             \nabla)u_{\lambda}
             +\nabla h'(n_{\lambda})
 =\nabla\phi_{\lambda}+\frac{1}{n_\lambda}\c B_{\lambda}\times B_{\lambda}, \label{mhd22}\\
&\pt B_{\lambda}-\c(u_{\lambda}\times B_{\lambda})=0,\label{mhd23}\\
&\md B_{\lambda}=0,\label{mhd24}\\
&\lambda^2\Delta\phi_{\lambda}=n_{\lambda}-1-\lambda^2\Delta
Q,\label{mhd25}
\end{align}
with $Q=-p^0.$ From Proposition \ref{local}, we have
\begin{align*}
    \sup_{t\in[0,T_\ast]}(\|\nabla Q(\cdot,t)\|_{s+1}+\|\pt \nabla
   Q(\cdot,t)\|_s)<+\infty.
\end{align*}

So, we can rewrite \eqref{mhd2111}-\eqref{mhd25} as
\begin{align}
    &\pt W_\lambda+\sum\limits_{i=1}^{3}A_i(W_\lambda)\partial_{x_i}W_\lambda
    =\left(
                   \begin{array}{c}
                     0 \\
                     \nabla\phi_\lambda  \\
                     0 \\
                   \end{array}
                 \right),\label{mhd21}\\
    &\lambda^2\Delta\phi_{\lambda}=n_\lambda-1-\lambda^2\Delta Q,\md B^\lambda=0.\label{mhd22}
\end{align}

 Set $$E=W^\lambda-W_\lambda=\left(
                                          \begin{array}{c}
                                            N \\
                                            U \\
                                            H \\
                                          \end{array}
                                        \right)=\left(
                                                  \begin{array}{c}
                                                    n^\lambda-n_\lambda \\
                                                    u^\lambda-u_\lambda \\
                                                    B^\lambda-B_\lambda \\
                                                  \end{array}
                                                \right),  \Phi=\lambda(\phi^\lambda-\phi_\lambda).
$$
%Recalling $p(n)=\frac{n^\gamma}{\gamma}$ and $n_\lambda=1$, we know
%that $D_0(n_\lambda)$ is a unit matrix.
We deduce from \eqref{mep1}-\eqref{mep2} and
\eqref{mhd21}-\eqref{mhd22} that
\begin{align}
    &\pt
    E+\sum\limits_{i=1}^{3}A_i(W^\lambda)\partial_{x_i}E
    =\frac{1}{\lambda}G+\sum\limits_{i=1}^{3}(A_i(W^\lambda)-A_i(W_\lambda))
    \partial_{x_i}W_\lambda,\label{error1}\\
    &\lambda\Delta\Phi=N+\lambda^2\Delta Q,\label{error2}
\end{align}
where $$G=\left(
                                  \begin{array}{c}
                                   0 \\
                                    \nabla\Phi \\
                                    0 \\
                                  \end{array}
                                \right).
$$

We differentiate \eqref{error1} with $\partial^{\alpha}_{x}$  for a
multi-index $\alpha$ satisfying $|\alpha|\leq s$ with $s>\frac52$ to
get
\begin{align}\label{erra}
    &\pt
    E_\alpha+\sum\limits_{i=1}^{3}A_i(W^\lambda)\partial_{x_i}E_\alpha=\frac{1}{\lambda}G_\alpha
    +R^1_\alpha+R^2_\alpha
\end{align}
with $\partial^{\alpha}_{x}f=f_\alpha,$ where
$$R^1_\alpha=\sum\limits_{i=1}^{3}[(A_i(W^\lambda)-A_i(W_\lambda))\partial_{x_i}W_\lambda]_\alpha,$$
$$R^2_\alpha=\sum\limits_{i=1}^{3}[A_i(W^\lambda)\partial_{x_i}E_\alpha-(A_i(W^\lambda)\partial_{x_i}E)_\alpha].$$

For the sake of clarity, we divide the following arguments into
lemmas.
\begin{lem}\label{lem5}
    Under the conditions of Theorem \emph{\ref{thm}}, we have
\begin{align}\label{lem51}
   &\ft\int_{\r}(E^T_\alpha
   A_0(n^\lambda)E_\alpha+|\nabla\Phi_\alpha|^2)
dx\notag\\&\qquad\leq C\|\nabla\Phi\|_s^2
    +\frac{C}{\lambda}\|U\|_s\|N\|_s\|\nabla\Phi_\alpha\|+
\frac{2}{\lambda}\|E\|_s\|U_\alpha\|\| \nabla\Phi_\alpha\|\notag\\
&\quad\quad\quad+C\|E_\alpha\|\|R^1_\alpha\|+C\|E_\alpha\|\|R^2_\alpha\|+C(1+\|E\|_s)\|E_\alpha\|^2+\lambda^2,
\end{align}
where $C$ is a generic constant depending only on the range
$(\frac12,\frac32)$of $n^\lambda$.
\end{lem}
\begin{proof}
  Taking the $L^2$ inner product of \eqref{erra}
with $D_0(n^\lambda)E_\alpha $, one gets, by integration by parts,
that
\begin{align}\label{211}
\ft\int_{\r}E^T_\alpha A_0(n^\lambda)E_\alpha dx&=
\frac{2}{\lambda}\int_{\r}E^T_\alpha A_0(n^\lambda)G_\alpha
dx+2\int_{\r}E^T_\alpha A_0(n^\lambda)R^1_\alpha
dx\notag\\
&\quad+2\int_{\r}E^T_\alpha A_0(n^\lambda)R^2_\alpha
dx\notag\\
&\quad+\int_{\r}E^T_\alpha \md A(W^\lambda)E_\alpha
dx\notag\\
&=\mathcal {I}^1_\alpha+\mathcal {I}^2_\alpha+\mathcal
{I}^3_\alpha+\mathcal {I}^4_\alpha
\end{align}
with $\md A(W^\lambda)=\pt
A_0(n^\lambda)+\sum\limits_{i=1}^{3}\partial_{x_{i}}(A_0(n^\lambda)A_i(W^\lambda)).$
Recalling that $$A_0(n^\lambda)=\left(
           \begin{array}{ccccccc}
             h'(n^\lambda) & 0 & 0 \\
             0 & n^\lambda\textbf{I}_{3\times3} & 0  \\
             0 & 0 & n^\lambda\textbf{I}_{3\times3} \\
           \end{array}
         \right)$$
and $n^\lambda=n_\lambda +N=1+N$, it is obvious that
\begin{align*}
    \mathcal {I}^1_\alpha&=\frac{2}{\lambda}\int_{\r}n^\lambda E^T_\alpha G_\alpha
dx\notag\\
&= \frac{2}{\lambda}\int_{\r}n^\lambda U_\alpha\cdot
\nabla\Phi_\alpha
dx\notag\\
&= \frac{2}{\lambda}\int_{\r}(1+N)U_\alpha\cdot \nabla\Phi_\alpha
dx\notag\\
&=-\frac{2}{\lambda}\int_{\r}\md U_\alpha\cdot \Phi_\alpha
dx+\frac{2}{\lambda}\int_{\r}NU_\alpha\cdot \nabla\Phi_\alpha
dx\notag\\
&\leq -\frac{2}{\lambda}\int_{\r}\md U_\alpha\cdot \Phi_\alpha
dx+\frac{2}{\lambda}\|N\|_\infty\int_{\r}|U_\alpha||
\nabla\Phi_\alpha|
dx\notag\\
&\leq-\frac{2}{\lambda}\int_{\r}\md U_\alpha\cdot \Phi_\alpha dx+
\frac{2}{\lambda}\|N\|_s\|U_\alpha\|\| \nabla\Phi_\alpha\|.
\end{align*}
Recalling $n_\lambda=1$ and $N=n^\lambda-n_\lambda$, from
\eqref{epm1} and \eqref{mhd21}, we have
\begin{align*}
    \md U=-\pt N-\md(u^\lambda N).
\end{align*}
Then, from \eqref{error2} we have
\begin{align*}
    -\frac{2}{\lambda}\int_{\r}\md U_\alpha\cdot \Phi_\alpha dx
    &=\frac{2}{\lambda}\int_{\r} \pt N_\alpha\cdot \Phi_\alpha dx
    -\frac{2}{\lambda}\int_{\r} (u^\lambda N)_\alpha\cdot \nabla\Phi_\alpha
    dx\\
    &= \frac{2}{\lambda}\int_{\r} (\lambda\pt \Delta\Phi_\alpha-\lambda^2\pt\Delta Q_\alpha)\cdot \Phi_\alpha dx
    -\frac{2}{\lambda}\int_{\r} (u^\lambda N)_\alpha\cdot
    \nabla\Phi_\alpha dx\\
    &=-\ft\int_{\r}|\nabla\Phi_\alpha|^2dx
    +2\lambda\int_{\r} \pt\nabla Q_\alpha\cdot \nabla\Phi_\alpha dx
    \\&\quad
    -\frac{2}{\lambda}\int_{\r} (U N)_\alpha\cdot
    \nabla\Phi_\alpha dx -\frac{2}{\lambda}\int_{\r} (u_\lambda N)_\alpha\cdot
    \nabla\Phi_\alpha dx\\
    &\leq-\ft\int_{\r} |\nabla\Phi_\alpha|^2dx
    +2\lambda\|\pt\nabla Q_\alpha\|\|\nabla\Phi_\alpha\|
    \\&\quad
    +\frac{C}{\lambda}\|U\|_s\|N\|_s\|\nabla\Phi_\alpha\|
    -\frac{2}{\lambda}\int_{\r} (u_\lambda N)_\alpha\cdot
    \nabla\Phi_\alpha dx\\
    &\leq -\ft\int_{\r} |\nabla\Phi_\alpha|^2dx
    +C\|\nabla\Phi_\alpha\|^2+\lambda^2
    +\frac{C}{\lambda}\|U\|_s\|N\|_s\|\nabla\Phi_\alpha\|
    \\&\quad -\frac{2}{\lambda}\int_{\r} (u_\lambda N)_\alpha\cdot
    \nabla\Phi_\alpha dx.
\end{align*}
Using \eqref{error2} and $\md u_\lambda=0$, we have, by part by
integrate, that
\begin{align*}
    -\frac{2}{\lambda}\int_{\r} (u_\lambda N)_\alpha\cdot
    \nabla\Phi_\alpha dx&=-2\int_{\r} (u_\lambda \Delta\Phi)_\alpha\cdot
    \nabla\Phi_\alpha dx-2\lambda\int_{\r} (u_\lambda \Delta Q)_\alpha\cdot
    \nabla\Phi_\alpha dx\\
    &\leq-2\int_{\r} (u_\lambda \Delta\Phi)_\alpha\cdot
    \nabla\Phi_\alpha dx+C\|\nabla\Phi_\alpha\|^2+\lambda^2\\
    &=-2\int_{\r} u_\lambda \Delta\Phi_\alpha\cdot
    \nabla\Phi_\alpha dx+C\|\nabla\Phi_\alpha\|^2+\lambda^2\\
    &\quad-2\int_{\r}[ (u_\lambda \Delta\Phi)_\alpha-u_\lambda \Delta\Phi_\alpha]\cdot
    \nabla\Phi_\alpha dx\\
    &=2\int_{\r}\nabla u_\lambda:(\nabla\Phi_\alpha\otimes \nabla\Phi_\alpha) dx
    +C\|\nabla\Phi_\alpha\|^2+\lambda^2\\
    &\quad-2\int_{\r}[ (u_\lambda \Delta\Phi)_\alpha-u_\lambda \Delta\Phi_\alpha]\cdot
    \nabla\Phi_\alpha dx\\
    &\leq C\|\nabla\Phi\|_s^2+\lambda^2.
\end{align*}
  where we have used the formulation
  $$\Delta\Phi_\alpha\cdot
    \nabla\Phi_\alpha=\md(\nabla\Phi_\alpha\otimes \nabla\Phi_\alpha)-\frac12\nabla|\nabla\Phi_\alpha|^2.$$

 Then we can show that
\begin{align}\label{i1a}
    \mathcal {I}^1_\alpha &\leq -\ft\int_{\r} |\nabla\Phi_\alpha|^2dx
    +C\|\nabla\Phi\|_s^2
    +\frac{C}{\lambda}\|U\|_s\|N\|_s\|\nabla\Phi_\alpha\|\notag\\&\quad+
\frac{2}{\lambda}\|E\|_s\|U_\alpha\|\|
\nabla\Phi_\alpha\|+\lambda^2.
\end{align}

For $\mathcal {I}^2_\alpha$ and $\mathcal {I}^3_\alpha$,  they  are
simply estimated as
\begin{align}
   \mathcal {I}^2_\alpha&=C\int_{\r}|E_\alpha||R^1_\alpha|
dx\leq C\|E_\alpha\|\|R^1_\alpha\|,\label{i2a}\\
\mathcal {I}^3_\alpha&=C\int_{\r}|E_\alpha||R^1_\alpha| dx\leq
C\|E_\alpha\|\|R^2_\alpha\|.\label{i3a}
\end{align}
Moreover, we have
$$|\md A(W^\lambda)|\leq C(1+\|E\|_s).$$
Then $\mathcal {I}^4_\alpha$ can be estimated as
\begin{align}\label{i4a}
    \mathcal {I}^4_\alpha\leq \|\md A(W^\lambda)\|_\infty\int_{\r}E^T_\alpha E_\alpha
dx\leq C(1+\|E\|_s)\|E_\alpha\|^2.
\end{align}
Now, substituting the inequalities \eqref{i1a}-\eqref{i4a} into
\eqref{211} gives \eqref{lem51}.
\end{proof}

Set $$\mathcal {D}=\mathcal
{D}(t)=\frac{\|E\|_s+\|\nabla\Phi\|_s}{\lambda}.$$ Then, for the inequality in Lemma
\ref{lem5}, we have the following claim.
 \begin{lem}\label{lemma6} For any $ \lambda\in(0,1)$, we have
\begin{align}\label{lem61}
   &\ft\int_{\r}(E^T_\alpha
   A_0(n^\lambda)E_\alpha+|\nabla\Phi_\alpha|^2)
dx\leq C(1+\mathcal
{D}^s)(\|E\|^2_{s}+\|\nabla\Phi\|^2_s)+\lambda^2.
\end{align}
\end{lem}
\begin{proof}
It is obviously that
\begin{align}
    &\frac{C}{\lambda}\|U\|_s\|N\|_s\|\nabla\Phi_\alpha\|\leq C
\mathcal {D}(\|E\|_s^2+\|\nabla\Phi\|^2_s),\label{171}\\
&\frac{2}{\lambda}\|E\|_s\|U_\alpha\|\| \nabla\Phi_\alpha\|\leq C
\mathcal {D}(\|E\|_s^2+\|\nabla\Phi\|^2_s),\label{172}\\
&(1+\|E\|_s)\|E_\alpha\|^2\leq C(1+\|\mathcal
{D}\|^s_s)\|E\|^2_s.\label{173}
\end{align}
Next we estimate $\|E_\alpha\|\|R^1_\alpha\|$. We use  the
boundedness of $\|(n_\lambda,u_\lambda,B_\lambda)\|_{s+1}
=\|(1,u^0,B^0)\|_{s+1}$ indicated in Proposition \ref{local}  to
conclude that
\begin{align*}
    \|R^1_\alpha\|&\leq C\sum^{3}_{i=1}\|u^\lambda_i-u_{\lambda i}\|_{|\alpha|}
    \|\partial_{x_{i}}W_\lambda\|_s+C\sum^{3}_{i=1}\|h'(n^\lambda)-h'(n_{\lambda})\|_{|\alpha|}
    \|\partial_{x_{i}}W_\lambda\|_s\\
    &\quad+C\sum^{3}_{i=1}\|n^\lambda-n_{\lambda}\|_{|\alpha|}
    \|\partial_{x_{i}}W_\lambda\|_s
    +C\sum^{3}_{i=1}\left\|\frac{G^\lambda_i}{n^\lambda}
    -\frac{G_{\lambda i}}{n_\lambda}\right\|_{|\alpha|}
    \|\partial_{x_{i}}W_\lambda\|_s\\
    &\leq C\|U\|_{|\alpha|}+C(1+\mathcal
    {D}^s)\|N\|_{|\alpha|}+C\|N\|_{|\alpha|}+C(1+\mathcal
    {D}^s)(\|N\|_{|\alpha|}+\|B\|_{|\alpha|})\\
    &\leq C(1+\mathcal
    {D}^s)\|E\|_s.
\end{align*}
In a similar spirit, $\|R^2_\alpha\|$  is estimated as
\begin{align*}
    \|R^2_\alpha\|&\leq C\sum^{3}_{i=1}\|u^\lambda_i\partial_{x_{i}}E_\alpha
    -(u_{\lambda i}\partial_{x_{i}}E)_\alpha\|\\&\quad
   +C\sum^{3}_{i=1}\|h'(n^\lambda)\partial_{x_{i}}N_\alpha-(h'(n^{\lambda})\partial_{x_{i}}N)_\alpha\|\\
    &\quad+C\sum^{3}_{i=1}\|n^\lambda\partial_{x_{i}}U_\alpha
    -(n^{\lambda}\partial_{x_{i}}U)_\alpha\|
    \\&\quad+C\sum^{3}_{i=1}\left\|\frac{(G^\lambda_i)^T}{n^\lambda}\partial_{x_{i}}H_\alpha
    -\left(\frac{(G^{\lambda}_{i} )^T}{n_\lambda}
    \partial_{x_{i}}H\right)_\alpha\right\|\\
    &\quad+C\sum^{3}_{i=1}\left\|\frac{G^\lambda_i}{n^\lambda}\partial_{x_{i}}U_\alpha
    -\left(\frac{G^{\lambda}_{i}}{n_\lambda}
    \partial_{x_{i}}U\right)_\alpha\right\|
    \\
    &\leq
    C\sum^{3}_{i=1}\|u^\lambda_i\|_s\|\partial_{x_{i}}E\|_{|\alpha|-1}
   +C\sum^{3}_{i=1}\|h'(n^\lambda)\|_s\|\partial_{x_{i}}N\|_{|\alpha|-1}\\
  &\quad +C\sum^{3}_{i=1}\|n^\lambda\|_s\|\partial_{x_{i}}U\|_{|\alpha|-1}
    +C\sum^{3}_{i=1}\left\|\frac{(G^\lambda_i)^T}{n^\lambda}\right\|_s\|\partial_{x_{I}}H\|_{|\alpha|-1}
    \\
    &\quad+C\sum^{3}_{i=1}\left\|\frac{G^\lambda_i}{n^\lambda}\right\|_s\|\partial_{x_{i}}U\|_{|\alpha|-1}
   \\&\leq C\sum^{3}_{i=1}(\|u^\lambda_i-u_{\lambda i}\|_s
   +\|u_{\lambda i}\|_s)\|E\|_{|\alpha|}
   +C(1+\|N\|_s^s)\|N\|_{|\alpha|}\\
  &\quad +C(\|n^\lambda-n_\lambda\|_s+\|n_\lambda\|_s)\|U\|_{|\alpha|}
    +C(1+\|E\|_s^s)\|E\|_{|\alpha|}
    \\&\leq C(1+\mathcal {D}^s)\|E\|_{s}.
\end{align*}This completes
the proof of Lemma \ref{lemma6}.
\end{proof}

Note that $C^{-1}\|E_\alpha\|^2\leq \int_{\r} E^T_\alpha
   A_0(n^\lambda)E_\alpha
dx\leq C \|E_\alpha\|^2$.  We integrate \eqref{lem61} from $0$ to
$t$ with $[0,t]\subset[0,\min\{T_\lambda,T_{\ast\ast}\}]$ to obtain
\begin{align*}
    \|E_\alpha(t)\|^2+\|\nabla \Phi_{\alpha}(t)\|^2\leq C\int_0^t(1+\mathcal {D}^s)
    (\|E\|^2_{s}+\|\nabla\Phi\|^2_s)d\tau+C\lambda^2.
\end{align*}
Here we have used the fact that the initial data constructed in
Theorem \ref{thm}. Summing up the last inequality over all ¦Á
satisfying $|\alpha|\leq s$, we get
\begin{align}
    \|E(t)\|^2_s+\|\nabla \Phi(t)\|^2_s\leq C\int_0^t(1+\mathcal {D}^s)
    (\|E\|^2_{s}+\|\nabla\Phi\|^2_s)d\tau+CT_{\ast\ast}\lambda^2.\label{179}
\end{align}
Applying Gronwall's lemma to \eqref{179},we obtain
\begin{align}
    \|E(t)\|^2_s+\|\nabla \Phi(t)\|^2_s
    \leq CT_{\ast\ast}\lambda^2e^{C\int_0^t(1+\mathcal {D}^s)d\tau}.
    \label{180}
\end{align}
In view of  $\|E\|_s+\|\nabla\Phi\|_s=\lambda\mathcal {D}$, it
follows from \eqref{180} that
\begin{align}
    D^2(t)
    \leq CT_{\ast\ast}e^{C\int_0^t(1+\mathcal {D}^s)d\tau}\equiv\Gamma(t).
    \label{181}
\end{align}
Thus, it holds that
\begin{align*}
   \Gamma'(t)
    = C(1+\mathcal {D}^s)\Gamma(t)\leq C\Gamma(t)+C\Gamma^{\frac{s+2}{2}}(t).
\end{align*}

Applying the nonlinear Gronwall-type inequality in \cite{y99} to the
last inequality yields $$\Gamma(t)\leq e^{CT_{\ast\ast}}$$ for
$t\in[0,\min\{T_\lambda,T_{\ast\ast}\}]$ if we choose $T_{\ast\ast}>0$ (independent of $\lambda$) so small
that
$$\Gamma(0)=CT_{\ast\ast}<e^{-CT_{\ast\ast}}.$$
Then, because of \eqref{181}, there exists a positive constant $M$,
independent of $\lambda$, such that
\begin{align}\label{186}
\mathcal {D}(t)\leq M
\end{align}
  for $t\in[0,\{T_\lambda,T_{\ast\ast}\}]$.
Finally,  from \eqref{180},\eqref{186} and the definition of
$(E,\nabla\Phi)$, we conclude the proof of Theorem \ref{thm}.

% Acknowledgments-------------------------------------------------------

\vspace{.2in} \noindent { Acknowledgments :}

 J. Yang's research was partially
supported by the Joint Funds of the National Natural Science
Foundation of China (Grant No. U1204103).
%--------------------------------------------------------------------

%\newpage

\end{document}